\documentclass[11pt,a4paper]{amsart}

\title{Non-associative Frobenius algebras of type $^1E_6$ with trivial Tits algebras}
\author{Jari Desmet}
\address{\parbox{\linewidth}{Ghent University \\ Department of Mathematics: Algebra and Geometry \\ Krijgslaan 281 -- S25 \\ 9000 Gent \\ Belgium}}
\email{\href{mailto:jari.desmet@ugent.be}{jari.desmet@ugent.be}}
\date{\today}
\makeatletter
\@namedef{subjclassname@2020}{%
  \textup{2020} Mathematics Subject Classification}
\makeatother

\usepackage[english]{babel}
\usepackage{amsmath,amssymb,amsthm}
\usepackage[utf8]{inputenc}
\usepackage{graphicx}
\usepackage{tikz}
\usetikzlibrary{calc}

\usepackage{arydshln}
\usepackage{tikzpagenodes}
\usepackage[shortlabels]{enumitem}
\setenumerate[0]{label={\rm (\roman*)}, leftmargin=*}
\usepackage{faktor}
\usepackage{mathtools}
\usepackage{breqn}
\usepackage{comment}

\usepackage[margin = 3cm]{geometry}
\usepackage[textsize=footnotesize,textwidth=15ex]{todonotes}

\usepackage{hyperref}

\usepackage[capitalise, noabbrev]{cleveref}

\theoremstyle{definition}
\newtheorem{definition}{Definition}[section]
\newtheorem{notation}[definition]{Notation}

\theoremstyle{plain}
\newtheorem*{theorem*}{Theorem}
\newtheorem{theorem}[definition]{Theorem}
\newtheorem{lemma}[definition]{Lemma}
\newtheorem{proposition}[definition]{Proposition}
\newtheorem{corollary}[definition]{Corollary}

\theoremstyle{remark}

\newtheorem{remark}[definition]{Remark}

\DeclareMathOperator{\spank}{span}
\DeclareMathOperator{\Lie}{Lie}
\DeclareMathOperator{\Der}{Der}
\DeclareMathOperator{\Stab}{Stab}
\DeclareMathOperator{\im}{Im}

\DeclareMathOperator{\gl}{\mathfrak{\gl}}
\DeclareMathOperator{\Oct}{\mathbb{O}}

\DeclareMathOperator{\Tr}{Tr}
\newcommand{\g}{\mathfrak{g}}

\DeclareMathOperator{\kar}{char}
\DeclareMathOperator{\ad}{ad}
\DeclareMathOperator{\Sq}{S}
\DeclareMathOperator{\End}{End}
\DeclareMathOperator{\Aut}{Aut}
\DeclareMathOperator{\GL}{GL}

\DeclareMathOperator{\Out}{Out}
\DeclareMathOperator{\Inn}{Inn}
\DeclareMathOperator{\Hom}{Hom}
\DeclareMathOperator{\orth}{orth}

\newcommand{\upi}[2]{#1^{(#2)}}

\newcommand{\id}{\mathrm{id}}

\numberwithin{equation}{section}

\begin{document}
\maketitle

\begin{abstract}
    Very recently, Maurice Chayet and Skip Garibaldi have introduced a class of commutative non-associative algebras, for each simple linear algebraic group over an arbitrary field (with some minor restriction on the characteristic).
    In a previous paper, we gave an explicit description of these algebras for groups of type $G_2$ and $F_4$ in terms of the octonion algebras and the Albert algebras, respectively. In this paper, we attempt a similar approach for type $E_6$.

\end{abstract}
\section{Introduction}

In \cite{chayet2020class}, Maurice Chayet and Skip Garibaldi define a construction with input any absolutely simple algebraic group $G$ over a field of large enough characteristic, and output a non-associative algebra $A(G)$ with a homomorphism $G\rightarrow \Aut(A(G))$. This construction applies to all absolutely simple algebraic groups, regardless of type and twisted form, and produces different algebras for different algebraic groups up to isogeny.

This construction was originally used to produce a new $3875$-dimensional algebra, the automorphism group of which is precisely a group of type $E_8$. Moreover, this algebra is the unique one on the second smallest representation for $E_8$ (the smallest representation is the adjoint one). Another construction for this algebra in type $E_8$ was found almost simultaneously by Tom De Medts and Michiel Van Couwenberghe in \cite{DMVC21}, where they prove complementary results about the structure of this algebra in the context of \emph{axial algebras}.

This class of algebras is explicitly constructed from the symmetric square of the Lie algebra. However, it would be interesting to find constructions of these algebras that use other algebraic structures instead of the adjoint representation. This is hinted at in \cite[Proposition~10.5]{chayet2020class}, but other than type $A_2$ in \cite[Example~10.9]{chayet2020class}, Chayet and Garibaldi do not investigate this further.

In \cite{jariprevious}, we provide an answer to this question for the algebras of type $G_2$ and $F_4$, where the algebras were constructed using the octonions and the Albert algebra respectively. In this paper, we use the Albert algebra to construct $A(E_6)$ for certain forms of $E_6$ (specifically, those of inner type with trivial Tits algebras, i.e.\@ the groups of type $E_6$ we can associate an Albert algebra with). Moreover, we provide some additional observations on the automorphism group of such algebras.

The paper follows a similar method of solving the problem as \cite{jariprevious}. However, a few additional complexities arise: as a simply connected group of type $E_6$ has center $C_3$ which acts non-trivially on the Albert algebra, we will need to cancel out that action on the algebra $A(G)$. This is solved by looking at the \emph{symmetric cube} $\Sq^3 W$ of the Albert algebra. It is then a question of finding a map from $\Sq^3W$ to $\sigma(A(G))$. Afterwards, we need to define multiplication maps on our model for $\sigma(A(G)$, which we will do by defining multiplication maps on $\Sq^3W$ that descend nicely onto $\sigma(A(G))$.

One of these multiplications is easily defined (see \cref{prop:mult1}), but another linearly independent one is harder to find. It turns out that we can cleverly use outer automorphisms of $E_6$ to produce a new multiplication (see \cref{prop:mult2}). In fact, this last method allows us to describe all non-commutative multiplication maps as well.

In \cref{preliminaries} we recall the results of \cite{chayet2020class}, and provide the necessary background to relate groups of type $E_6$ with Albert algebras. Afterwards, in \cref{sec:outaut}, we make some observations about outer automorphisms and how they can be modelled on representation of simple algebraic groups. Moreover, we prove in the case of $E_6$ that the automorphism group of an $E_6$-equivariant polynomial on $V(\omega_1+\omega_6)$ cannot be much larger than $E_6$ itself. 

\cref{sece6} is the largest section of the article: we provide a model for the representation $V(\omega_1+\omega_6)$ using the Albert algebra, and describe all (commutative) $E_6$-equivariant multiplications on this model explicitly. Afterwards, in \cref{identificationag2}, we use these results to construct a model for $A(G)$ using the Albert algebra.

Lastly, in \cref{sec:autgroup}, we observe that any commutative algebra defined on the representation $V(\omega_1+\omega_6)$ has automorphism group $G\rtimes C_2$ where $G$ is the adjoint group of type $E_6$, and that any non-commutative algebra on $V(\omega_1+\omega_6)$ has automorphism group $G$, using results from \cref{sec:outaut,sece6}.
\subsection*{Acknowledgements}
	The author is supported by the FWO PhD mandate 1172422N. The author is grateful to his supervisor Tom De Medts, for guiding him through the process of writing this article.
\subsection*{Assumptions}
We use the same restrictions on the characteristic as \cite{chayet2020class}, i.e.\@ $\kar k \geq h+2 = 14$ or $0$ with $h$ the Coxeter number of the Dynkin diagram $E_6$. In this case, the Weyl module of highest weight $\omega_1+\omega_6$ for $E_6$ is irreducible, as is the Weyl module of weight $2\tilde{\alpha}$, where $\tilde{\alpha}$ is the highest root, by \cite{lub01}. In particular, we can do character computations for these representations independent of the field $k$.

\section{Preliminaries}\label{preliminaries}
\subsection{Representation theory of absolutely simple groups and Lie algebras}
%\todoi{change this statement}
The irreducible representations of an absolutely simple split group over a field $k$ are classified by the dominant weights of the associated root system (\cite[Theorem~22.2]{Milne}). To any dominant weight $\lambda$ we can also associate a so-called \emph{Weyl module} $V(\lambda)$. Under the characteristic assumption above, we know the Weyl modules considered below for type $E_6$ are irreducible \cite{lub01}. We will often identify a representation by its associated dominant weight using labelling as in \cite[Plates I-IX, p.264-290]{Bourbaki46}. As mentioned in \cite[\S 7, p.10-11]{chayet2020class}, for non-split reductive groups there is a unique representation that becomes isomorphic to $V(\lambda)$ when base changing to $\bar{k}$ if $\lambda$ is contained in the root lattice and fixed by the Galois action as in \cite[Théorème 3.3]{Tits1971}. In the case of type $E_6$, this means we need the coefficients of $\omega_1$ and $\omega_6$ to be equal, as well as the coefficients of $\omega_3$ and $\omega_5$. We will denote that representation by the same notation. For further discussion on irreducible representations of simple algebraic groups, see \cite{jantzenrepalggroups}.

	We will work with the representations of algebraic groups as representations of their Lie algebras. We will use the same notation $V(\lambda)$ for the Weyl module when considered as a Lie algebra representation. 
	
%	\todoi{Beetje te veel herhaling in de vorige paragraaf? Misschien gewoon zo: \\
%	We will work with the representations of algebraic groups as representations of their Lie algebras. (Note, however, that not every representation of the Lie algebra corresponds to a representation of the associated algebraic group.) We will use the same notation $V(\lambda)$ for the Weyl module when considered as a Lie algebra representation. 
%	}

	We will use characters of representations of Lie algebras to compute dimensions of certain representations (\cite[Corollary~24.6]{FultonHarris}) and morphism spaces (the argument for finite groups is given in \cite[p.12]{FultonHarris}, but also holds for Lie algebras), as well as decompositions of symmetric powers (\cite[Exercise~23.39]{FultonHarris}).

\begin{remark}\label{reptheoryremark}
    Characters of representations of Lie algebras (in particular the Weyl modules, over algebraically closed fields) can be computed using Sage \cite{sagemath}.
\end{remark}
\subsection{Twisted forms of type $E_6$.}
An absolutely simple algebraic group is of type $E_6$ if it becomes isomorphic to an absolutely simple algebraic group of type $E_6$ when the base field is extended to an algebraic closure. This means that over arbitrary fields, we have a wider variety of groups of type $E_6$. As will become apparent in the rest of this paper, the representation theory of non-split simple groups depends on the way the group is twisted. This phenomenon was studied by Jaqcues Tits in \cite{Tits1971}. The concepts of \emph{Tits indices} (see \cite[Section~2.3]{tits1966classification}) and \emph{Tits algebras} (see \cite[Section 27]{involutions}) are essential in this work.

For our purposes, it is relevant to make the distinction between groups of \emph{inner type} $^1E_6$ and \emph{outer type} $^2E_6$. The absolutely simple group of type $E_6$ will be called of inner type when the Galois action (the $*$-action as in \cite[Section~2.3]{tits1966classification}) on the Dynkin diagram is trivial, and of outer type when it has order $2$. It is well known that we can only associate an Albert algebra to a group of type $E_6$ if it is of type $^1E_6$ and has trivial Tits algebras (see \cite[Theorem 1.4]{Garibaldie67}). As we use the Albert algebra to construct our new model for $A(G)$, we will have to restrict ourselves to this condition as well.

It remains to be seen if there is a way of extending this model to groups of type $^2E_6$, in a similar fashion as was done for Brown algebras and Freudenthal triple systems in \cite{Tomhermitiancubic}.

%\todoi{uithalen.}
%We will also need the following known result, see e.g.\ \cite[VIII.\S6.4]{Bourbaki7-9} (for the positive characteristic proof, see the slightly more general \cite[Lemma~2.9]{chayet2020class}).
%\begin{proposition}\label{casimirisscalar} Let $\g$ be a simple Lie algebra (associated to an algebraic group $G$) over $k$.
%	Let $\pi \colon \g \to \End(V)$ be equivalent to the Weyl module $V(\lambda)$ over an algebraic closure of $k$, with $\lambda$ a dominant weight. Denote by $\delta$ half the sum of the positive roots. Then:
%	\begin{enumerate}
%		\item $ \sum \pi(X_i)\pi(Y_i) = \langle \lambda ,  \lambda + 2 \delta \rangle \cdot \id_V $ for dual bases $\{X_i\}$ and $\{Y_i\}$ with respect to the Killing form,
%		\item for all $X,Y\in \g$ we have
%		$ \Tr(\pi(X)\pi(Y)) = \tfrac{\langle \lambda ,  \lambda + 2 \delta \rangle \cdot \dim V}{\dim \g}K(X,Y)$.
%	\end{enumerate}
%\end{proposition}	
\subsection{Construction of the algebras}

We summarize the construction of the Chayet-Garibaldi algebras in this short section. For more background, see \cite{chayet2020class,jariprevious}.

In their paper, Chayet and Garibaldi constructed an algebra for any absolutely simple algebraic group over a field of large enough characteristic, and showed the following. The characteristic condition ensures a counit can be defined, since $\dim G$ will be invertible. Recall that $h$ denotes the Coxeter number of the root system asociated to the simple linear algebraic group $G$.

\begin{theorem}[\cite{chayet2020class}]\label{modulestructure}
Let $G$ be an absolutely simple algebraic group over a field $k$ with $\kar k\geq h+2$ or $0$, and let $\g = \Lie(G)$. Define
\begin{align*}
	S\colon \Sq^2\g &\to \End(\g) \\
	XY&\mapsto h^{\vee} \ad X \bullet \ad Y + \tfrac{1}{2}(X K(Y,\mathunderscore) + Y K(X, \mathunderscore) )\text{,}
\end{align*}
and
\begin{align*}
	\diamond \colon \im(S)\times \im(S) &\to \im(S) \\
	(S(AB), S(CD)) \mapsto \tfrac{h^{\vee}}{2} ( &
		S(A, (\ad C \bullet \ad D)B) + S((\ad C \bullet \ad D)A,B) \\
		&+ S(C, (\ad A \bullet \ad B)D) + S( (\ad A \bullet \ad B)C, D)\\
		&+ S([A,C][B,D])+S([A,D][B,C]) )\\
		+ \tfrac{1}{4}(& K(A,C)S(BD) + K(A,D)S(BC)\\
		&+ K(B,C)S(AD) + K(B,D)S(AC) )\text{.}
\end{align*}
Then $A(G) \coloneqq (\im(S),\diamond)$ is a well-defined simple unital commutative non-associative algebra with counit $\varepsilon(a) \coloneqq \tfrac{1}{\dim(G)}\Tr(a)$.

If the type of $G$ appears in \cref{tablemodulestructure}, $A(G)$ decomposes as a representation of $G$ into $A(G) = k\oplus V(\lambda)$,  where $\lambda$ is as in \cref{tablemodulestructure}. The Weyl module $V(\lambda)$ occurring in this decomposition is irreducible.	
\end{theorem}

%\begin{lemma}[{\cite[Lemma~6.1]{chayet2020class}}]\label{associativitytau}
%	The bilinear form $\tau$ on $\Ag$ associates with $\diamond$, i.e.\ for all $ a,a',a'' \in \Ag$,
%	\[ \tau(a\diamond b,c) = \tau(a, b\diamond c) . \]
%\end{lemma}
%In \cite{chayet2020class}, the module structure of $A(\g)$ was completely identified as well.
%%\todoi{verwoording in \cite{chayet2020class} over de karakteristiek was echt verwarrend}
%\begin{proposition}[{\cite[Proposition 7.2]{chayet2020class}}]\label{modulestructure}
%	Let $G$ be an absolutely simple algebraic group $G$ appearing in \cref{tablemodulestructure} and assume that $\kar k =0$ or $\kar k \geq h+2 $. As a representation of $G$, we have $\Ag = k \oplus V(\lambda)$, where $\lambda$ is as in \cref{tablemodulestructure}, and the Weyl module $V(\lambda)$ is irreducible in.
%\end{proposition}
%
\begin{table}[h]
	\centering
	\begin{tabular}{c|c c c c c c}
		type of $G$& $A_2$ & $G_2$ & $F_4$ & $E_6$ & $E_7$&$E_8$ \\
		\hline
		dual Coxeter number $h^{\vee}$ & $3$&$4$&$9$&$12$&$18$&$30$ \\
		Coxeter number $h$ & $3$&$6$&$12$&$12$&$18$&$30$ \\
		Dominant weight $\lambda$ & $\omega_1 + \omega_2$&$2\omega_1$&$2\omega_4$&$\omega_1+\omega_6$&$\omega_6$&$\omega_1$ \\
		Dimension of $V(\lambda)$ & $8$&$27$&$324$&$650$&$1539$&$3875$ \\
	\end{tabular}
	\medskip
	
	\caption{The table from \cite{chayet2020class}. The fundamental dominant weights are labelled using Bourbaki labelling.}\label{tablemodulestructure}
\end{table}

In the last section of \cite{chayet2020class}, the authors describe an embedding of these algebras into the endomorphism ring of a natural representation for the groups of type $A_2,G_2,F_4,E_6$ and $E_7$. It is this embedding that we will use to obtain new descriptions for the algebras in question.
\begin{proposition}[{\cite[Proposition~10.5]{chayet2020class}}]\label{embedding}
	If $G$ has type $A_2,G_2,F_4,E_6$ or $E_7$ and $\pi\colon G \to \GL(W)$ is the natural irreducible representation of dimension $3,7,26,27$ or $56$ respectively, then the formula
	\begin{equation}\label{sigma}
		\sigma(S(XY)) =6h^{\vee} \pi(X)\bullet \pi(Y) - \tfrac{1}{2}K(X,Y)
	\end{equation}
	defines an injective $G$-equivariant linear map
	\[\sigma \colon A(\g) \hookrightarrow \End(W)\text{.} \]
	Moreover, $\sigma$ maps the identity $\id_\g$ to the identity $\id_W$.	
\end{proposition}

The cases $A_2,G_2,F_4$ have already been explored previously, see \cite[Example 10.6]{chayet2020class} and \cite{jariprevious}. To deal with the case of $E_6$, we will model its $27$-dimensional irreducible representation as the exceptional \emph{Albert algebra}, the hermitian matrices over the octonion algebra. The following two subsections will introduce octonion algebras and Albert algebras.

\subsection{The octonion algebra}\label{octonionsubsection}

Though octonion algebras can be defined over any characteristic, we restrict ourselves to $\kar k \neq 2$. We will outline some results about these objects that will be of use later, mainly for computations.

The treatment given in this section is based on \cite{SpringerVeldkamp}.
\begin{definition}[Octonion algebra]\leavevmode
	\begin{enumerate}
		\item An \emph{octonion algebra} is an $8$-dimensional $k$-algebra $A$ equipped with a nondegenerate quadratic form $N\colon A \to k$ such that
	\[N(a b) = N(a)N(b),\]
	for all $a,b\in A$. We call $N$ the \emph{norm} of the composition algebra. We denote its associated bilinear form by $\langle \cdot, \cdot \rangle$, so $\langle x,y \rangle = N(x+y)-N(x) - N(y)$ for any two octonions $x,y \in A$. We will say two octonions $a,b \in A$ are orthogonal and write $a\perp b$ whenever $\langle a,b \rangle =0$. If $N$ is isotropic, we will call the octonion algebra \emph{split}.
		\item Let $A$ be an octonion algebra with identity $e$.
	We define the \emph{standard involution} $\overline{\cdot}\colon A \to A$ by
	\[ \overline{x} = \langle x,e \rangle e -x, \] for all $x\in A$. The standard involution is an anti-automorphism of the octonion algebra.
	\end{enumerate}
\end{definition}

%\todoi{check which identities are in fact necessary}
%\begin{proposition}\label{octonionidentities} For $x,y,z \in A$ we have the following identities:
%	\begin{enumerate}\itemsep1ex
%		\item\label{minimimalpoly}
%		$xy+yx - \langle x,e\rangle y - \langle y,e\rangle x +\langle x, y \rangle e =0 $, 
%
%		\item\label{1.2.6} 
%	$		\langle xy,z \rangle = \langle y, \overline{x}z \rangle\text{, }
%		\langle xy,z \rangle = \langle x, z\overline{y} \rangle\text{, }
%		\langle xy,\overline{z} \rangle = \langle yz, \overline{x}\rangle\text{,}
%	$
%		\item\label{normrule}
%	$
%		x(\overline{x}y) = N(x)y\text{, }
%		(x\overline{y})y = N(y)x\text{.}
%	$
%		
% 		\item\label{moufangidentities} (Moufang identities): 
%	$
%		(zx)(yz) = z((xy)z)\text{, }
%		z(x(zy)) = (z(xz))y\text{, }
%		x(z(yz)) = ((xz)y)z\text{.}
%	$
%	\end{enumerate}
%\end{proposition}
%\begin{proof}
%	These are \cite[Proposition~1.2.3, Lemma~1.3.2, Lemma~1.3.3 and Proposition~1.4.1]{SpringerVeldkamp}, respectively.
%\end{proof}
%The octonion algebras are examples of so-called \emph{alternative algebras}, i.e. the associator $\{x,y,z\} \coloneqq (xy)z-x(yz)$ satisfies the property
%\[ \{x_{\pi(1)},x_{\pi(2)},x_{\pi(3)}  \} = \sgn(\pi) \{x_1,x_2,x_3\}\]
%for every permutation $\pi \in \mathcal{S}_3$.

It will be convenient to make use of a standard basis to deal with some computations. 
\begin{proposition}\label{orthogonalanisotropic}
	Any octonion algebra $A$ over a field $k$ with $\kar k \neq 2$ has an orthogonal basis of the form $e_0=e,e_1 = a,e_2 = b,e_3 = ab, e_4 = c,e_5=ac,e_6 = bc,e_7 = (ab)c$, with $N(a)N(b)N(c)\neq0$.
\end{proposition}
\begin{proof}
	See \cite[Corollary~1.6.3]{SpringerVeldkamp}.
\end{proof}
\begin{remark}\label{standardbasis}
	In case $k$ is algebraically closed, we can assume $N(a)=N(b)=N(c)=1$, and we call such a basis a \emph{standard basis} for the octonions. For a basis of this form, the multiplication is encoded by the Fano plane (see \cite[Figure~1]{jariprevious}). We will use this basis in explicit computations.
\end{remark}
It is a well-known fact that up to isomorphism, there is only one octonion algebra over an algebraically closed field, and only one split one over arbitrary fields.
\begin{proposition}
	Any two split octonion algebras over a field $k$ are isomorphic.
\end{proposition}
\begin{proof}
	See \cite[Theorem 1.8.1]{SpringerVeldkamp}.
\end{proof}

\subsection{The Albert algebras}\label{albertsubsection} We will only consider the Albert algebras over fields $k$ with $\kar k\neq 2,3$.
	In this case, the $27$-dimensional representation for type $E_6$ can be constructed as matrices over the octonions (at least for some groups of this type, see \cref{prop:E6trivialtits} and the introduction of \cref{sece6}). We give a short overview in this subsection. 
	
	We regard the split Albert algebra as the hermitian matrices $\mathcal{H}_3{(\mathbb{O})}$, where $\Oct$ denotes the split octonions over a field $k$. We write\footnote{The notation was revised slightly compared to \cite{jariprevious} to avoid confusing index notation.}
\[ \mathbf{1}_1 \coloneqq \begin{bsmallmatrix}
	1&0&0\\
	0&0&0\\
	0&0&0
\end{bsmallmatrix},\mathbf{1}_2 \coloneqq \begin{bsmallmatrix}
	0&0&0\\
	0&1&0\\
	0&0&0
\end{bsmallmatrix},\mathbf{1}_3 \coloneqq \begin{bsmallmatrix}
	0&0&0\\
	0&0&0\\
	0&0&1
\end{bsmallmatrix}, \]
and for octonions $a,b,c\in \Oct$
\[ \upi{a}{1} \coloneqq \begin{bsmallmatrix}
	0&0&0\\
	0&0&a\\
	0&\bar{a}&0
\end{bsmallmatrix}, \upi{b}{2} \coloneqq \begin{bsmallmatrix}
	0&0&\bar{b}\\
	0&0&0\\
	b&0&0
\end{bsmallmatrix},\upi{c}{3} \coloneqq \begin{bsmallmatrix}
	0&c&0\\
	\bar{c}&0&0\\
	0&0&0
\end{bsmallmatrix}. \]	
	A generic element of $\mathcal{H}_3{(\mathbb{O})}$ is then of the form
\[  
\begin{bsmallmatrix}
	\alpha_1 & c & \bar{b} \\
	\bar{c} & \alpha_2 & a \\
	b & \bar{a} & \alpha_3
\end{bsmallmatrix}
= \alpha_1\mathbf{1}_1+\alpha_2\mathbf{1}_2+\alpha_3\mathbf{1}_3+\upi{a}{1}+\upi{b}{2}+\upi{c}{3}\text{,}
\]

with $\alpha_1,\alpha_2,\alpha_3\in k$ and $a,b,c\in \mathbb{O}$. To keep the notation uniform with the octonion algebras, $e = \mathbf{1}_1+\mathbf{1}_2+\mathbf{1}_3$ denotes the unit of an Albert algebra.

\begin{definition}\label{def:albertalg}\leavevmode
	\begin{enumerate}
		\item We define the split Albert algebra to be the hermitian matrices over the split octonions $\mathcal{H}_3(\mathbb{O})$, equipped with the Jordan product, i.e. for $a,b\in \mathcal{H}_3(\Oct)$:
			\[ a\cdot b \coloneqq \tfrac{ab+ba}{2}\text{,} \]
			where the multiplication on the right-hand side is the usual matrix multiplication.
		
		\item We define a nondegenerate bilinear form $\langle \cdot,\cdot \rangle\colon \mathcal{H}_3(\mathbb{O}) \times \mathcal{H}_3(\mathbb{O}) \to k $ by the formula
		\[\langle x,y \rangle  \coloneqq \Tr(xy)\text{.}\]
		We will sometimes refer to this bilinear form as the \emph{trace form}.
		
		\item A not-necessarily split Albert algebra is a $k$-algebra equipped with a bilinear form $\langle \cdot,\cdot\rangle_A $ and a trilinear form $\langle \cdot,\cdot,\cdot \rangle_A$ such that there exists an isomorphism of algebras $A\otimes_k\overline{k}\cong \mathcal{H}_3(\mathbb{O})$ sending the extension of the bilinear form $ \langle \cdot,\cdot\rangle_A $ to $\langle \cdot,\cdot \rangle$ and the extension of the trilinear form $\langle \cdot,\cdot,\cdot \rangle_A$ to the trilinearisation $\langle \cdot,\cdot,\cdot \rangle$ of the usual determinant form $\det$ on the $3\times3$ matrices over $\Oct$.
		
			We will omit writing the index $A$ throughout this article.
	\end{enumerate}	
\end{definition}
\begin{remark}
	The trilinearisation of the cubic $\det$ is divided by a scalar $6$, i.e.\@ \[ \langle x,y,z \rangle = \tfrac{1}{6}( \det(x+y+z) - \det(x+y) -\det(y+z) - \det(x+z) + \det(x) + \det(y)+ \det(z) ), \]
	for all $x,y,z\in A$.
\end{remark}
\begin{remark}
	We will write both the trace form on the Albert algebra and the bilinear form on the octonions as $\langle\cdot,\cdot\rangle$. This should not cause any confusion however, as for two octonions $a,b\in \mathbb{O}$ and $i,j\in \{1,2,3\}$ we have $\langle \upi{a}{i},\upi{b}{j}\rangle = \delta_{i,j}\langle a,b\rangle$ where the $\langle \cdot ,\cdot \rangle$ in the left hand side and right hand side are bilinear forms in different algebras.
\end{remark}
We immediately have some routine computations we will use regularly.
\begin{lemma} \label{albertmultiplication} Using the notation introduced above, the following equations hold in the Albert algebra over an algebraically closed field $k$:
	\begin{enumerate}
		\item $\mathbf{1}_i^2 = \mathbf{1}_i$ for $i\in \{1,2,3\}$,
		\item $\mathbf{1}_i\cdot \mathbf{1}_j = 0$ for $i\neq j$ and $i,j\in \{1,2,3\}$,
		\item $\mathbf{1}_i \cdot \upi{a}{i} = 0$ for $i\in \{1,2,3\}$ and $a\in \mathbb{O}$,
		\item $\mathbf{1}_i \cdot \upi{a}{j} = \tfrac{1}{2}\upi{a}{j}$ for $i,j\in \{1,2,3\}$ with $i\neq j$ and $a\in \Oct$,
		\item $\upi{a}{i}\cdot \upi{b}{i}$ = $\tfrac{1}{2}\langle a,b\rangle(\mathbf{1}_j+\mathbf{1}_k)$ for $\{i,j,k\} = \{1,2,3\}$ and $a,b\in \Oct$,
		\item $\upi{a}{i}\cdot \upi{b}{j} = \tfrac{1}{2}\upi{\left(\overline{ab}\right)}{k}$ for $i,j,k$ a cyclic permutation of $1,2,3$.
	\end{enumerate}
\end{lemma}
\begin{proof}
	These are \cite[Identities~(4.26)-(4.31)]{nonassocalg}.
\end{proof}
%The following lemma will be useful to simplify some computations.
%\begin{lemma}\label{equation512}
%	For $x,y,z \in A$ with $A$ an Albert algebra over a field $k$, we have 
%	\begin{multline*}
%		x\cdot(y\cdot z)+y\cdot (x\cdot z)+z\cdot (x\cdot y) \\ = \langle x,e \rangle y\cdot z +\langle y,e \rangle x\cdot z + \langle z,e \rangle x\cdot y 
%		+\tfrac{1}{2}( \langle x,y\rangle -\langle x,e\rangle \langle y,e\rangle)z\\+\tfrac{1}{2}( \langle x,z\rangle -\langle x,e\rangle \langle z,e\rangle)y 
%		+\tfrac{1}{2}( \langle y,z\rangle -\langle y,e\rangle \langle z,e\rangle)x+ 3\langle x,y,z\rangle e\text{.}
%	\end{multline*}
%\end{lemma}
%\begin{proof}
%	This is \cite[Equation~5.12]{SpringerVeldkamp}.
%\end{proof}
We will also need the so-called cross product, as it has a special relation with respect to groups of type $E_6$.
\begin{definition}[{\cite[Equation~5.16]{SpringerVeldkamp}}]
	For $x,y\in A$ elements of an Albert algebra, we define $x\times y\in A$ to be the unique element such that
	\[3\langle x,y,z\rangle  = \langle x\times y,z\rangle \]
	holds for all $z\in A$.
\end{definition}
\begin{lemma}\label{lem:crossproductproperties}
	\begin{enumerate}
		\item The stabiliser $G$ of the determinant is a simply connected group of type $E_6$.
		\item Taking the transpose $^{\top}$ with respect to the bilinear form $\langle\cdot,\cdot\rangle$ and then inverting is an order $2$ outer automorphism of $G$.
		\item For $g\in G(k)$ and $x,y\in A$ we have $x^g\times y^g=(x\times y)^{g^{-\top}}$.
		\item We have $x\times y = xy - \tfrac{1}{2}\langle x,e \rangle y- \tfrac{1}{2}\langle y,e \rangle x- \tfrac{1}{2}\langle x,y \rangle e + \tfrac{1}{2}\langle x,e \rangle\langle y,e\rangle e$. 
	\end{enumerate}
\end{lemma}
\begin{proof}
	See \cite[Lemma~5.2.1 and Section~7.3]{SpringerVeldkamp}.
\end{proof}
In fact, we have a converse direction as well. For the definition of Tits algebras, see \cite[Section~27]{involutions}.
\begin{proposition}\label{prop:E6trivialtits}
	Any simply connected group of type $^1E_6$ with trivial Tits algebras occurs as the stabiliser of the cubic norm of an Albert algebra.
\end{proposition}
\begin{proof}
	See \cite[Theorem 1.4]{Garibaldie67}.
\end{proof}
\begin{proposition}\label{prop:derivationsstructure}
	The space of derivations of the cubic form $\det$ on the Albert algebra is given by
	\[ \Der(\det) = \spank_k \{ L_a, [L_a,L_b] \mid a,b \in W \text{ such that } \langle a,e\rangle = \langle b,e\rangle = 0\}. \]
\end{proposition}
\begin{proof}
	This is \cite[Theorem~4.12]{nonassocalg}.
\end{proof}

\begin{proposition}\label{prop:maxweightisrank1}
	Let $v^+\in W$ be the maximal weight vector with respect to a certain choice of maximal torus and Borel subgroup of $G$. Then $v^+\times v^+=0$. Moreover, $G$ acts transitively on the elements $v\in W$ with $v\times v = 0$. We call such elements rank $1$ elements.
\end{proposition}
\begin{proof}
	See for example \cite[Subsection 7.10]{garibaldi2006geometries}.
\end{proof}
\section{Outer automorphisms on representations of type $E_6$}\label{sec:outaut}

The Dynkin diagram of $E_6$ has non-trivial symmetries, which sets it apart from the other exceptional types. This \emph{outer automorphism} allows for some interesting phenomena to occur. For example, the split Lie algebra of type $E_6$ has as automorphism group the Chevalley group of type $E_6$, extended by a $C_2$ symmetry, which corresponds to the $C_2$ symmetry of the Dynkin diagram $\Delta$.

For a representation $V$, we could ask if there is a linear map $V\to V$ that models this outer automorphism, i.e.\@ an element in $\GL(V)$ such that conjugation induces the outer automorphism on $G$. %The following section will deal with this question, though it is probably already known by experts, see for example \cite[p.34]{GG15}. 
This question was already answered in \cite[Proposition~2.2]{linpres}. The proposition tells us that the outer automorphisms we can model are precisely those that fix the highest weight of $V$.

\begin{proposition}[{\cite[Proposition~2.2]{linpres}}]
	Let $G$ be a simple algebraic group, and $V = V(\lambda)$ an irreducible representation on which it acts faithfully. Then the normaliser of $G$ in $\GL(V)$ is smooth, and can be identified with $((\mathbb{D}\times  G)/Z)\rtimes \Aut(\Delta,\lambda)$, where $\Aut(\Delta,\lambda)$ is the automorphism group of the associated Dynkin diagram $\Delta$ that fixes $\lambda$, $\mathbb{D}$ is the group diagonal matrices, and $Z$ the center of $G$.
\end{proposition}

The following statement can be easily deduced from \cite[Proposition~2.2]{linpres}. 

\begin{proposition}\label{prop:outeriffselfdual}
	An irrep $V$ of a simple algebraic group $G$ of type $E_6$ admits an outer automorphism if and only if it is self dual.
\end{proposition}
\begin{proof}
	Let $\rho$ denote the representation $\rho\colon V\to \GL(V)$.
	Let $\phi$ be an outer automorphism of $G$. We can assume it has order $2$, since $\Aut(G)$ is a split extension of $\Inn(G)$ and $\Out(G)$ by \cite[Corollary~23.47]{Milne}.
	Fix a maximal torus $T$ and Borel subgroup $B$. Suppose $V$ has highest weight $\lambda$ and outer automorphism $\varphi$ of type $\phi$. Then $V$ should be isomorphic to the irrep with highest weight $\phi(\lambda)$, where $\phi$ denotes the action of the outer automorphism induced on the Dynkin diagram (i.e.\@ the basis of the root system). By \cite[p.\@ 471]{Milne}, the dual representation of $V$ has highest weight $-w_0(\lambda)$, where $w_0\in W$ is the unique element for which $w_0$ which interchanges the positive and negative roots (this exists by \cite[Summary~21.41]{Milne}). We know precisely in which way $w_0$ acts (\cite[Plate (V), item (XI)]{Bourbaki46}). It is then clear that $\phi(\lambda) = -w_0(\lambda)$, thus $V$ is isomorphic to its dual representation.
	
	Conversely, suppose $V$ is self dual. Define a new action on $V$ by $g\cdot v  \coloneqq \rho(\phi(g))(v)$. Then $V$ with this representation of $G$ is also irreducible, and by the previous paragraph, there exists an isomorphism $\varphi$ between the two module structures. \end{proof}
\begin{corollary}\label{thm:extension}
	Let $G$ be a simple algebraic group over an infinite field $k$, $V=V(\lambda)$ an irrep, and $f\colon V^{\otimes r}\otimes {(V^*)}^{\otimes s}\to k$ a linear map. Assume $G\trianglelefteq \Stab(f)$, and that $\Stab(f)$ is smooth. Then there exists a subgroup $H\leq \Aut(\Delta)$ such that \[ \Stab(f)/(G\cdot \mu_{r-s}) \cong H .\] 
\end{corollary}
\begin{proof}
	 The stabiliser $\Stab(f)$ has to be contained in the normaliser of $f$, which can be identified with $((\mathbb{D}\times  G)/Z)\rtimes \Aut(\Delta,\lambda)$. It only remains to observe that scalars preserve $f$ if and only if they are $(r-s)$th roots of unity.
\end{proof}
%\begin{proof}
%	Any scalar matrix stabilising $f$ should be in $\mu_{r-s}$ and vice versa. Because $G\trianglelefteq \Stab(f)$ we know any $\varphi\in \Stab(f)$ should be a group automorphism. Now define 
%		\begin{align*}
%		\Stab(f)/(G\cdot \mu_{r-s}) &\to \mathcal{S}(\Delta) \\
%		[\varphi] &\mapsto t(\varphi).	
%		\end{align*}
%		From \cref{prop:outerautounique}, this map is a well defined injective group homomorphism.
%\end{proof}

\begin{proposition}\label{prop:transposeisouterauto}
	Let $V$ be the $27$-dimensional Albert algebra. Then the transpose operator $\top \colon \End(V) \to \End(V)\colon \varphi \mapsto \varphi^{\top}$ with respect to the bilinear form $\langle\cdot , \cdot \rangle$ is an order $2$ outer automorphism on $\End(V)$.
\end{proposition}
\begin{proof}
	Let $\rho\colon G \to \GL(V)$ denote the representation. Then the $G$-action on $\End(V)$ is given by $\varphi^{\rho(g)} = \rho(g)\varphi\rho(g)^{-1}$. Thus $(\varphi^{\rho(g)})^\top = \rho(g)^{-\top}\varphi^\top \rho(g)^{\top} = (\varphi^\top)^{\rho(g)^{-\top}}$ . The result now follows from \cref{lem:crossproductproperties}(ii). 
\end{proof}

This means that to find $\Stab(f)$ for maps $f\colon V^{\otimes r}\otimes {(V^*)}^{\otimes s}\to k$, where $V$ is a self dual representation, all we need to do is check that a certain outer automorphism preserves the map $f$. This might or might not be the case: we will see examples of both cases in \cref{sec:autgroup}. %There is something interesting going on: self dual representations are contained in the root lattice of $E_6$, and are fixed by the non-trivial Galois action as in \cite[Theorème~3.3]{Tits1971}. This means there is a unique representation for any form of $E_6$ (outer or inner) that becomes isomorphic to a self dual irrep of given weight $\lambda$. It would be interesting to find a polynomial on such an irrep that has automorphism group precisely $E_6$, or prove that they cannot exist. In \cref{thm:allmultc2}, we prove this is not the case for commutative algebra structures on $V(\omega_1+\omega_6)$.
%\todoi{ask tom about perhaps an argument.}
\section{The algebra of type $^1E_6$}\label{sece6}

Recall that
 the assumption on the characteristic in this section is $\kar k =0$ or $\kar k>13$. We will write $G$ for the group of type $E_6$ used to construct the algebra $A(G)$.
 
 We need to assume three things about the group $G$ to make sure we can associate an Albert algebra to it. First of all, the action of the absolute Galois group needs to fix $\omega_1$, otherwise we cannot speak of the representation $V(\omega_1)$. This comes down to considering groups of type $^1E_6$, i.e.\@ the $*$-action \cite[Section~2.3]{tits1966classification} of the absolute Galois group is trivial on the root lattice of $G$. Second, the group $G$ has to be simply connected. We can assume this without loss of generality, by if necessary taking its universal cover as in \cite[Remark~23.60]{Milne}. The groups of type $^1E_6$ occur as norm forms of Albert algebras precisely when they have \emph{trivial Tits algebras}. This is the last assumption we need.
 
 When $G$ is of type $^2E_6$, there is no way of distinguishing between $\omega_1$ and $\omega_6$, by \cite[Theorème 7.2]{Tits1971}. In fact, there is no $27$-dimensional irreducible representation, and we cannot use \cite[Proposition~10.5]{chayet2020class}. In these cases, we need to extend to a quadratic field extension to obtain a group of type $^1E_6$ and possibly extend the base field even further to make sure the Tits algebra is trivial.

Throughout the rest of this section, we will assume $k$ is algebraically closed, with the exception of \cref{identificationag2}.

In this part, $W$ denotes the $27$-dimensional $G$-representation with highest weight $\omega_1$, and $V$ denotes the $650$-dimensional representation with highest weight $\omega_1+\omega_6$. In this setting, the embedding from \cref{embedding} becomes, using \cite[Lemma~2.9]{chayet2020class},
\begin{equation}\label{embeddingg2}
	\sigma(S(XY)) = 72X\bullet Y - 2\Tr(XY)\id_{W}\text{.} 
\end{equation}
In this formula, we do not write the embedding $\pi$ of the Lie algebra explicitly.

Our model for the algebra is constructed as a subrepresentation of the endomorphism ring of the Albert algebra, thus we will need to introduce notation for the operators we will use.
\begin{definition}\label{def:operators}
Let $x,y,z\in W$ be Albert elements. 
	\begin{enumerate}
		\item We write $L_x$ for the operator $L_x\colon W \to W \colon a \mapsto x\cdot a$,
		\item we write $xy$ for the operator $xy\colon W \to W \colon a \mapsto \tfrac{1}{2}(\langle a,x\rangle y + \langle a, y \rangle x)$,
		\item for $i\in \{1,2,3\}$ we define $I_i \coloneqq \sum_{x\in B} x_ix_i$, where $B$ is an orthonormal basis for the octonions.
		\item we write $xyz$ for the operator $xyz\colon W \to W \colon a \mapsto \langle a,x,y\rangle z + \langle a, y,z \rangle x + \langle a, z,x \rangle y $.
	\end{enumerate}
\end{definition}
In the following, we will use a specific expression of the identity in terms of these newly defined operators:
\begin{lemma}\label{lem:idexpress}
We have the following equality of operators on $W$:
	\begin{multline*}
		\id_W = 3\big(2\mathbf{1}_1\mathbf{1}_2\mathbf{1}_3 \\+ \upi{(e_0)}{1}\upi{(e_0)}{2}\upi{(e_0)}{3}+\upi{(e_1)}{1}\upi{(e_3)}{2}\upi{(e_2)}{3}+\upi{(e_2)}{1}\upi{(e_6)}{2}\upi{(e_4)}{3}+\upi{(e_3)}{1}\upi{(e_5)}{2}\upi{(e_6)}{3}\\+\upi{(e_4)}{1}\upi{(e_1)}{2}\upi{(e_5)}{3}+\upi{(e_5)}{1}\upi{(e_2)}{2}\upi{(e_7)}{3}+\upi{(e_6)}{1}\upi{(e_7)}{2}\upi{(e_1)}{3}+\upi{(e_7)}{1}\upi{(e_4)}{2}\upi{(e_3)}{3}\big).
	\end{multline*}
\end{lemma}
\begin{proof}
	Note that for any triple $\{x_1,x_2,x_3\}$ occuring in the sum on the right hand side, we have $x_i\times x_j = \tfrac{1}{2}x_k$ for $\{i,j,k\} = \{1,2,3\}$. This means the right hand side is equal to $\sum_{1\leq i \leq 3} \Big(\mathbf{1}_i\mathbf{1}_i + \tfrac{1}{2}\sum_{0\leq j\leq 7} \upi{(e_j)}{i}\upi{(e_j)}{i}\Big)$. This is clearly equal to the identity, since $B=\{\mathbf{1}_i,\tfrac{1}{\sqrt{2}}\upi{(e_j)}{i} \mid 1\leq i\leq 3, 0\leq j\leq 7\}$ is an orthornormal basis for the Albert algebra.
\end{proof}

The next lemma is necessary to find the image of $L_xL_y \in \Sq^2\Lie(G)$ using \cref{embeddingg2}.
\begin{lemma}\label{eq512operator}
For $x,y\in W$ traceless Albert elements, we have
	\[ L_x\bullet L_y = -\tfrac{1}{2}L_{x\cdot y} +\tfrac{1}{2} \langle e,\cdot \rangle x\cdot y+ \tfrac{1}{4}\langle x,y\rangle \id_W  + \tfrac{1}{2}xy +\tfrac{1}{2}\langle x \times y,\cdot \rangle e.\]
\end{lemma}
\begin{proof}
	This is \cite[Equation~(5.12)]{SpringerVeldkamp}.
\end{proof}
It turns out we can very easily find the full image of $\sigma$ using tools from representation theory.
\begin{proposition} Recall the notation from \cref{def:operators}. We have
	\[ \sigma(A(G)) = \spank_k\{ a_1a_2a_3 \mid a_1,a_2,a_3 \in W \}\text{.} \]
\end{proposition}
\begin{proof}
	The symmetric cube of $W$ decomposes as a representation in $\Sq^3W \cong A(G) \oplus V(3\omega_1)$. The representation $V(3\omega_1)$ is generated by a maximal weight vector of the form $v^+v^+v^+$, where $v^+$ is a maximal weight vector of $W$. By \cref{prop:maxweightisrank1}, we have $v^+\times v^+ =0$. This means that $V(3\omega_1)$ is in the kernel of the map
	\begin{align*}
		\pi\colon \Sq^3W &\to \End(W)\colon\\
		a_1a_2a_3&\mapsto  a_1a_2a_3\text{,}
	\end{align*}
	since $v^+v^+v^+ = v^+\langle v^+\times v^+,\cdot \rangle =0$.
	It is then easy to verify that the image of this map is precisely $\sigma(A(G))$.
\end{proof}
\begin{remark}\label{rem:identif}
	Note that the previous proof allows us to identify $\sigma(A(G))$ with $\Sq^3W/V(3\omega_1)$, using the map $\pi(a_1a_2a_3 + V(3\omega_1) = a_1a_2a_3$. We will do so frequently in the next subsection.
	
	This identification also tells us we can find a basis for $\sigma(A(G))$ consisting of operators of the form $a_1a_2a_3$ with $a_1,a_2,a_3\in W$. We use this fact to simplify the notation in the next subsection.
\end{remark}

We want to find the product for $A(G)$ in terms of the operators $a_1a_2a_3\in \End(W)$. Thus we link the product $\diamond$ on $A(G)$ to the operators in question. To do this, we will use the explicit form of the Lie algebra of type $E_6$ as derivations of the cubic form on the Albert algebra given in \cref{prop:derivationsstructure}.
%\begin{proposition}
%	Let $W$ be an Albert algebra. The derivations of the cubic form (as in \cref{def:albertalg}~(iii)) form a Lie algebra of type $E_6$, and is spanned by $L_a,[L_a,L_b]$, where $a,b$ are traceless albert elements. 
%\end{proposition}
%\begin{proof}
%	This is \cite[Theorem~4.12]{nonassocalg}.
%\end{proof}
\begin{definition}
	Let $a,b\in \sigma(A(G))$. We define the multiplication $\star$ by $a\star b \coloneqq \sigma(\sigma^{-1}(a)\diamond \sigma^{-1}(b))$.
\end{definition}

\begin{proposition}\label{prop:examplecomp}
	Let $a,b\in \Oct$ be octonions such that $a\cdot a = b\cdot b =0$ and $\langle a,b\rangle \neq 0$. Then for $i\in \{1,2,3\}$ we have
	\[ \sigma(S(L_{\upi{a}{i}}L_{\upi{a}{i}})) = -108\upi{a}{i}\upi{a}{i}e, \]
	and
	%\[ \sigma(S(L_{a_i}L_{a_i})\diamond S(L_{b_i}L_{b_i}) ) = 216(I_j+I_k)+864a_ib_i-144\id_W. \]
	%So in fact
	\[\upi{a}{i}\upi{a}{i}e\star \upi{b}{i}\upi{b}{i}e = \tfrac{1}{216}(I_j+I_k)\langle \upi{a}{i},\upi{b}{i}\rangle^2+\tfrac{1}{27} \langle \upi{a}{i},\upi{b}{i}\rangle \upi{a}{i}\upi{b}{i}-\tfrac{1}{324} \langle \upi{a}{i},\upi{b}{i} \rangle^2 \id_W. \]
\end{proposition}
\begin{proof}
	By \cref{eq512operator}, we have for $x,y\in e^{\perp}$ traceless Albert elements that $L_x\bullet L_y = -\tfrac{1}{2}L_{x\cdot y} + \frac{1}{2}\langle e, \cdot \rangle (x\cdot y) + \tfrac{1}{4}\langle x,y\rangle \id_W+ \tfrac{1}{2}xy + \tfrac{3}{2}\langle x,y,\cdot\rangle e$. This means that 
	\begin{align*}
		\sigma(S(L_{\upi{a}{i}}L_{\upi{a}{i}})) &= 72L_{\upi{a}{i}}^2 - 2\Tr(L_{\upi{a}{i}}^2)\id_W\\
		&= 72(\tfrac{1}{2}\upi{a}{i}\upi{a}{i} + \tfrac{3}{2}\langle \upi{a}{i},\upi{a}{i},\cdot \rangle e)\text{.}
	\end{align*}
	Note however that $\upi{a}{i}\times \upi{a}{i} = 0$, and $\upi{a}{i}\times e = -\tfrac{1}{2}\upi{a}{i}$, so 
	\begin{align*}
		\sigma(S(L_{\upi{a}{i}}L_{\upi{a}{i}})) &= 72(-\tfrac{3}{2}\upi{a}{i}\upi{a}{i}e ) = -108\upi{a}{i}\upi{a}{i}e.
	\end{align*}
	The same of course holds for $\upi{b}{i}$. We then compute 
	\begin{multline}\label{eq:prop:examplecomp}
		\sigma(S(L_{\upi{a}{i}}L_{\upi{a}{i}})\diamond S(L_{\upi{b}{i}}L_{\upi{b}{i}}) )
		\\ = h^\vee \sigma(S([L_{\upi{a}{i}},[L_{\upi{a}{i}},L_{\upi{b}{i}}]] L_{\upi{b}{i}})) + h^\vee \sigma(S([L_{\upi{b}{i}},[L_{\upi{b}{i}},L_{\upi{a}{i}}]] L_{\upi{a}{i}}))\\
		+h^\vee \sigma(S([L_{\upi{a}{i}},L_{\upi{b}{i}}][L_{\upi{a}{i}},L_{\upi{b}{i}}])) + K(L_{\upi{a}{i}},L_{\upi{b}{i}}))S(L_{\upi{a}{i}}L_{\upi{b}{i}})).
	\end{multline}
	Note however that $D_{\upi{a}{i},\upi{b}{i}} = [L_{\upi{a}{i}},L_{\upi{b}{i}}]$ is a derivation of the Albert algebra (\cite[Identity~(4.6)]{nonassocalg}), so $[L_{\upi{a}{i}},[L_{\upi{a}{i}},L_{\upi{b}{i}}]] = -L_{D_{\upi{a}{i},\upi{b}{i}}(\upi{a}{i})}$, which is equal to $-\tfrac{1}{2}\langle a,b\rangle L_{\upi{a}{i}}$. So \cref{eq:prop:examplecomp} becomes
	\begin{multline}\label{eq:prop:examplecomp2}
		\sigma(S(L_{\upi{a}{i}}L_{\upi{a}{i}})\diamond S(L_{\upi{b}{i}}L_{\upi{b}{i}}) )
		\\ =  -\tfrac{1}{2}h^\vee\langle a,b\rangle \sigma(S(L_{\upi{a}{i}} L_{\upi{b}{i}})) - \tfrac{1}{2} h^\vee\langle a,b\rangle\sigma(S(L_{\upi{a}{i}} L_{\upi{b}{i}}))\\
		+h^\vee \sigma(S([L_{\upi{a}{i}},L_{\upi{b}{i}}][L_{\upi{a}{i}},L_{\upi{b}{i}}])) + 4\Tr(L_{\upi{a}{i}}\bullet L_{\upi{b}{i}})S(L_{\upi{a}{i}}L_{\upi{b}{i}})).
	\end{multline}
	Using \cref{eq512operator}, we can show that $\Tr(L_{\upi{a}{i}}\bullet L_{\upi{b}{i}})= 3\langle \upi{a}{i},\upi{b}{i}\rangle$, so \cref{eq:prop:examplecomp2} becomes
	\begin{multline}\label{eq:prop:examplecomp3}
		\sigma(S(L_{\upi{a}{i}}L_{\upi{a}{i}})\diamond S(L_{\upi{b}{i}}L_{\upi{b}{i}}) )
		\\ =  -\tfrac{12}{2}\langle a,b\rangle \sigma(S(L_{\upi{a}{i}} L_{\upi{b}{i}})) - \tfrac{12}{2}\langle a,b\rangle\sigma(S(L_{\upi{a}{i}} L_{\upi{b}{i}}))\\
		+12 \sigma(S([L_{\upi{a}{i}},L_{\upi{b}{i}}][L_{\upi{a}{i}},L_{\upi{b}{i}}])) + 12S(L_{\upi{a}{i}}L_{\upi{b}{i}})) \\
		= 12 \sigma(S([L_{\upi{a}{i}},L_{\upi{b}{i}}][L_{\upi{a}{i}},L_{\upi{b}{i}}])).
	\end{multline}
	By \cite[Lemma 5.2]{jariprevious}, this proves the formula in the statement of the proposition.
\end{proof}

\subsection{Defining multiplications on $V$}\label{sec:defmultg2}
Recall from \cref{modulestructure} that $A(G) \cong k\oplus V$, where $V$ is the irreducible $560$-dimensional representation of the group $G$ of type $E_6$. By \cite[Example~A.6]{chayet2020class}, multiplication is of the form
\begin{equation}\label{eq:multag2}
	(\lambda, u)\star (\mu,v) = (\lambda\mu+f(u,v), \lambda v+\mu u + u\odot v)\text{,} 
\end{equation}
for a certain $G$-invariant symmetric bilinear form $f$ and $G$-equivariant symmetric multiplication $\odot$ on $V$. By computing characters, we can see there is only a one dimensional space of symmetric invariant bilinear forms and only a two dimensional space of symmetric multiplications that are $G$-equivariant.
Note that the embedding $\sigma$ sends $V$ to the subspace of trace zero elements, i.e.\@
\[\sigma(V) = \left\{ \sum_{i} a_ib_ic_i \,\middle\vert\, \sum_{i} \langle a_i,b_i,c_i \rangle =0 \right\}\text{.} \]
To define commutative $G$-equivariant multiplications $\odot$ on $V\subseteq\sigma(A(G))$, we will do so first on $\Sq^3W\cong k\oplus V \oplus  V(3\omega_1)$. To make sure these induce well-defined multiplications onto $\sigma(A(G))$ (see \cref{rem:identif}), we require that \begin{equation}\label{eq:necessaryformult}
	v^+v^+v^+\odot a_1a_2a_3\in V(3\omega_1)
\end{equation} for any rank $1$	 element $v^+\in W$ and elements $a_1,a_2,a_3\in W$. This suffices, as $v^+v^+v^+$ is a maximal weight vector of $3\omega_1$ by \cref{prop:maxweightisrank1}. Indeed if $\odot\colon 
\Sq^2(\Sq^3W)\to \Sq^3W$ is a commutative multiplication satisfying \cref{eq:necessaryformult}, then $ \left( a_1a_2a_3+V(3\omega_1)\right)\odot'\left( b_1b_2b_3+V(3\omega_1)\right) \coloneqq \pi\left( a_1a_2a_3\odot b_1b_2b_3\right)$, with $\pi$ as in \cref{rem:identif}, is a well defined $G$-equivariant multiplication on $\sigma(A(G))$. 

We also know that the bilinear form in \cref{eq:multag2} will have to be a scalar multiple of the following form, restricted to $\sigma(V)$.
\begin{definition}\label{def:fprime}
	Define $f'\colon \Sq^3W \times \Sq^3W\to k$ by the formula \[f'(a_1a_2a_3,b_1b_2b_3) = \sum_{\tau_a,\tau_b\in \mathcal{S}_3} \langle a_{\tau_a(1)} , b_{\tau_b(2)} ,b_{\tau_b(3)}\rangle \langle b_{\tau_b(1)},a_{\tau_a(2)} ,a_{\tau_a(3)} \rangle . \]
\end{definition}

\begin{proposition}
	The bilinear form $f'$ factors through $\sigma(A(G))\times \sigma(A(G))$, and thus defines a non-zero symmetric $G$-equivariant bilinear form on $\sigma(A(G))$.
\end{proposition}
\begin{proof}
	It suffices to check that  $f'$ is symmetric, $G$-equivariant and $f'(v^+v^+v^+,a_1a_2a_3) = 0 $ for any rank $1$ element $v^+$ and $a_1,a_2,a_3\in W$. For any two octonions $a,b\in \mathbb{O}$ with $a\cdot a = b\cdot b = 0$, $ \langle a,b\rangle \neq 0$ and $i\in \{1,2,3\}$ we have
	\begin{equation}\label{eq:bilform}
		f'(\upi{a}{i}\upi{a}{i}\mathbf{1}_i,\upi{b}{i}\upi{b}{i}\mathbf{1}_i) = 16\langle \upi{a}{i} , b_{i} ,\mathbf{1}_i \rangle^2 = \tfrac{4}{9}\langle \upi{a}{i},\upi{b}{i}\rangle^2 
		= \tfrac{4}{9}\langle a,b\rangle^2\text{.} 
	\end{equation}
	As $\langle a,b\rangle$ is non-zero, this proves $f'$ is non-zero. 
\end{proof}
Using the above considerations, we now turn to finding a basis for the space of $G$-invariant multiplications.
\subsubsection{Finding multiplication 1}	 By \cref{lem:crossproductproperties}, the cross product is not $G$-equivariant, however the $4$-linear product $(x\times y)\times(z\times w)$ (for $x,y,z,w\in W$) is. Using this observation and the discussion above, we can immediately construct our first multiplication.
\begin{proposition}\label{prop:mult1}
%Let $\id_W = \sum_m x_{1,m}x_{2,m}x_{3,m}$. 
Define for $a_1a_2a_3,b_1b_2b_3\in \End(V)$
\begin{multline}\label{eq:def:odot1}
	a_1a_2a_3 \odot_1 b_1b_2b_3  \\
	\coloneqq
	\sum_{\tau_a,\tau_b \in \mathcal{S}_3}  \left(((a_{\tau_a(1)}\times a_{\tau_a(2)})\times(b_{\tau_b(1)}\times b_{\tau_b(2)}))a_{\tau_a(3)}b_{\tau_b(3)} \right) \\ - \tfrac{1}{12} f'(a_1,a_2,a_3,b_1,b_2,b_3)\id_{W}\text{.}  %- \tfrac{1}{27}\langle (a_{\tau_a(1)}\times a_{\tau_a(2)})\times(b_{\tau_b(1)}\times b_{\tau_b(2)}),a_{\tau_a(3)},b_{\tau_b(3)}\rangle  \id_{W}\text{.} 
\end{multline}
Then $\odot_1$ (after restricting to $\sigma(V)$) is a well-defined commutative $G$-equivariant product on $\sigma(V)$.
%Then the space of commutative $E_6$-equivariant products on $\sigma(V)$ is spanned by $\odot_1$ and $\odot_2$, restricted to $\sigma(V)$.
\end{proposition}	
\begin{proof}
	Note first that \cref{eq:necessaryformult} is satisfied. The product is clearly $G$-equivariant and commutative on $\sigma(A(G))$. To ensure that the product is well-defined when restricting to $\sigma(V)$, it suffices to check that the right-hand side of \cref{eq:def:odot2} is contained in the kernel of the trace map. But note that $\Tr(\cdot \odot_2|_{\sigma(V)} \cdot)$ is a $G$-invariant bilinear form, and thus equal to $f'$ up to a scalar. We prove this scalar has to be zero by the following computation, where $a,b$ are octonions with $a\cdot a=b\cdot b =0$ and $\langle a,b\rangle \neq 0$:
		\begin{multline}\label{eq:exodot1}
			\upi{a}{i}\upi{a}{i}\mathbf{1}_i\odot_1\upi{b}{i}\upi{b}{i}\mathbf{1}_i = 16((\upi{a}{i}\times \mathbf{1}_i)\times(\upi{b}{i}\times \mathbf{1}_i))\upi{a}{i}\upi{b}{i}- \tfrac{1}{12} f'(\upi{a}{i}\upi{a}{i}\mathbf{1}_i,\upi{b}{i}\upi{b}{i}\mathbf{1}_i)\\
			= 4(\upi{a}{i}\times \upi{b}{i})\upi{a}{i}\upi{b}{i}- \tfrac{1}{ 27}\langle a,b\rangle^2\id_W
			= -2\langle a,b\rangle\mathbf{1}_i\upi{a}{i}\upi{b}{i}- \tfrac{1}{ 27}\langle a,b\rangle^2\id_W\\
			= \tfrac{2}{3}\langle a,b\rangle \upi{a}{i}\upi{b}{i}  + \tfrac{1}{3}\langle a,b\rangle^2\mathbf{1}_i\mathbf{1}_i - \tfrac{1}{ 27}\langle a,b\rangle^2\id_W
		\end{multline}	
	Applying the trace form to the right-hand side of this equation, we obtain zero. As $f'(\upi{a}{i}\upi{a}{i}\mathbf{1}_i,\upi{b}{i}\upi{b}{i}\mathbf{1}_i)$ is non-zero  (see \cref{eq:bilform}), this proves $\Tr(\cdot \odot_1|_{\sigma(V)} \cdot)$ is zero, and thus $\odot_1|_{\sigma(V)}$ has image inside $\sigma(V)$.	
\end{proof}
\begin{remark}
	It is possible to check directly that the trace of the right-hand side of \cref{eq:def:odot1} is zero, but we will need \cref{eq:exodot1} later on, in \cref{prop:multbasis}.
\end{remark}

\subsubsection{Finding multiplication 2} For the second multiplication, the idea is to find an outer automorphism $\varphi\colon V\to V$ such that $\varphi(a^g) = \varphi(a)^{\phi(g)}$, and a multiplication $\boxdot$ such that $a^g\boxdot b^g = {(a\boxdot b)}^{\phi(g)}$. Then $\varphi\circ (\cdot  \boxdot \cdot  )$ is $G$-equivariant.

This idea works, but only if we first lift to the cubic tensor power of $W$, instead of working with the symmetric cube of $W$.
\begin{lemma}\label{lem:psi}
	Let $\psi \colon W^{\otimes3} \otimes \sigma(A(G)) \to \sigma(A(G))$
	be defined by
	\[ \psi(a_1\otimes a_2 \otimes a_3 ,x_1x_2x_3) = \sum_{\sigma \in \mathcal{S}_3} (a_1\times x_{\sigma(1)})(a_2\times a_3 ) (x_{\sigma(2)}\times x_{\sigma(3)}).  \]
	Then $\psi$ is well defined.
\end{lemma}
\begin{proof}
	We want to prove that for $v \in V(3\omega_1)$, we have $\psi(a_1\otimes a_2\otimes a_3,v) = 0 $ for all $a_1\otimes a_2\otimes a_3 \in W^{\otimes 3}$. Let $v^+$ be a maximal weight vector of $W$. Then $v^+v^+v^+$ is a maximal weight vector of $V(3\omega_1)$, and $\psi(a_1\otimes a_2\otimes a_3,v^+v^+v^+) = 0 $ for all $a_1\otimes a_2\otimes a_3 \in W^{\otimes 3}$. Now we have $\psi(a_1\otimes a_2\otimes a_3,{v^+}^g{v^+}^g{v^+}^g) = \psi(a_1^{g^{-1}}\otimes a_2^{g^{-1}}\otimes a_3^{g^{-1}},v^+v^+v^+)^{g^{-\top}} = 0$. As elements of the form ${v^+}^g{v^+}^g{v^+}^g$ with $g\in G$ span $V(3\omega_1)$, we are done.
\end{proof}

In fact, using this map we can obtain an outer automorphism. Note that because of the uniqueness of outer automorphisms on irreps, this has to be equal to the transposition operator (up to a scalar) on $V$, see \cref{thm:extension,prop:transposeisouterauto}.
\begin{proposition}\label{prop:outeraut}
	Let $\psi_{\id_W}\colon W^{\otimes 3} \to \sigma(A(G)) $ be the map defined by \[\psi_{\id_W}(a_1\otimes a_2\otimes a_3) = \psi(a_1\otimes a_2\otimes a_3,\id_W), \] and $P\colon \Sq^3 W \to W^{\otimes 3} $ by \[P(a_1a_2a_3) = \sum_{\tau_a \in \mathcal{S}_3} a_{\tau_a(1)}\otimes a_{\tau_a(2)}\otimes a_{\tau_a(3)}.\] Then the kernel of $\varphi = \psi_{\id_W}\circ P\colon \Sq^3W \to \sigma(A(G))$ contains $V(3\omega_1)$ and the natural projection $\varphi\colon  \sigma(A(G))\to \sigma(A(G)) $  is an outer automorphism, more specifically $\varphi(v^g) = \varphi(v)^{g^{-\top}}$ for $v\in \sigma(A(G))$.
\end{proposition}
\begin{proof}
	One can provide the exact same argument as in \cref{lem:psi} to prove that $V(3\omega_1)\subseteq \ker(\varphi)$. For the second claim, let $a_1a_2a_3\in \sigma(A(G))$ be arbitrary. Then 
	\begin{multline*}
		\varphi(a_1^ga_2^ga_3^g) = \psi\left(\sum_{\tau_a \in \mathcal{S}_3} a_{\tau_a(1)}^g\otimes a_{\tau_a(2)}^g\otimes a_{\tau_a(3)}^g, \id_W\right) \\=\psi\left(\sum_{\tau_a \in \mathcal{S}_3} a_{\tau_a(1)}\otimes a_{\tau_a(2)}\otimes a_{\tau_a(3)}, {\id_W}^{g^{-1}}\right)^{g^{-\top}} =  \psi\left(\sum_{\tau_a \in \mathcal{S}_3} a_{\tau_a(1)}\otimes a_{\tau_a(2)}\otimes a_{\tau_a(3)}, \id_W\right)^{g^{-\top}} \\=\varphi(a_1a_2a_3)^{g^{-\top}}. \qedhere
	\end{multline*} 
\end{proof}
\begin{corollary}
	The outer automorphism $\tfrac{1}{36}\varphi$ is an order $2$ automorphism on $\sigma(V)$. Moreover, it is equal to the restriction of the transposition operator on $\sigma(V)$ as in \cref{prop:transposeisouterauto}. 
\end{corollary}
\begin{proof}
We can use the expression for the identity in \cref{lem:idexpress}.	For $a$ an octonion with $a\cdot a = 0$ and $i\in \{1,2,3\}$, we can compute the value $\varphi(\upi{a}{i} \upi{a}{i} \mathbf{1}_i)$: 	\begin{align*}
		\varphi(\upi{a}{i}\upi{a}{i}\mathbf{1}_i) &=  4\sum_{\sigma \in \mathcal{S}_3} (\upi{a}{i}\times x_{\sigma(1)})(\upi{a}{i}\times \mathbf{1}_i ) (x_{\sigma(2)}\times x_{\sigma(3)})\\
		&= -2\sum_{\sigma \in \mathcal{S}_3} (\upi{a}{i}\times x_{\sigma(1)})(\upi{a}{i} ) (x_{\sigma(2)}\times x_{\sigma(3)}) \\
		&= -6\cdot 2 \Big((\upi{a}{i}\times \mathbf{1}_i)(\upi{a}{i} ) (\mathbf{1}_i) + \tfrac{1}{2} \sum_{\substack{1\leq j\leq 8 \\ k\in \{1,2,3\}}} (\upi{a}{i}\times {(e_j)}_k)(\upi{a}{i}) ({(e_j)}_k)\Big)\\
		&= 36\upi{a}{i}\upi{a}{i}\mathbf{1}_i
	\end{align*}
	As $\varphi^2$ is a scalar matrix by Schur's Lemma, the first claim follows. The second claim follows from the fact that $\varphi\circ \top$ is a $G$-equivariant endomorphism of $\sigma(V)$ by \cref{prop:transposeisouterauto}, and thus a scalar matrix. As it fixes $\upi{a}{i}\upi{a}{i}\mathbf{1}_i$, it has to equal the identity.
\end{proof}
In this proof, an expression occurred that we will see multiple times throughout. To shorten this expression, we will introduce some notation.
\begin{notation}\label{not:orth}
	Given an Albert element $x\in W$, we write
	\[\orth(x) \coloneqq \sum_{b\in B} (x\times b)xb .\]
	In this expression, $B$ denotes the orthonormal basis $\{\mathbf{1}_i, \tfrac{1}{\sqrt{2}}\upi{(e_j)}{i} \mid 1\leq i \leq 3, 1\leq j \leq 8\}$ of $W$.
\end{notation}
With this definition, we get that $\orth(\upi{a}{i}) = -3\upi{a}{i}\upi{a}{i}1_i$.

The previous corollary can make computations easier, though it is nice to describe all operators fully in terms of triples and cross products.

\begin{proposition}\label{prop:boxdot}
	Let $\cdot \boxdot \cdot \colon \sigma(A(G))\otimes \sigma(A(G)) \to W^{\otimes 3}$ be defined by 
	\[a_1a_2a_3 \boxdot b_1b_2b_3 \coloneqq \sum_{\tau_a,\tau_b \in \mathcal{S}_3} (a_{\tau_a(1)}\times  b_{\tau_b(1)})\otimes(a_{\tau_a(2)}\times a_{\tau_a(3)})\otimes(b_{\tau_b(2)}\times b_{\tau_b(3)}).  \]
	Then $\boxdot$ is well-defined, and $(a_1a_2a_3)^g \boxdot (b_1b_2b_3)^g = (a_1a_2a_3 \boxdot b_1b_2b_3)^{g^{-\top}}$.
\end{proposition}
\begin{proof}
	This is completely analogous to the proof of \cref{lem:psi}.
\end{proof}

\begin{proposition}\label{prop:mult2}
Let $\id_W = \sum_m x_{1,m}x_{2,m}x_{3,m}\in \End(W)$, and define for $a_1a_2a_3,b_1b_2b_3\in \End(V)$
	\begin{multline}\label{eq:def:odot2}
	a_1a_2a_3 \odot_2  b_1b_2b_3  \\ \coloneqq
	\sum_{\substack{m\\ \sigma \in \mathcal{S}_3}}\sum_{\tau_a, \tau_b\in \mathcal{S}_3}  \left(((a_{\tau_a(1)}\times b_{\tau_b(1)})\times x_{m,\sigma(1)}) ((a_{\tau_a(2)}\times a_{\tau_a(3)})\times(b_{\tau_b(2)}\times b_{\tau_b(3)}))(x_{m,\sigma(2)}\times x_{m,\sigma(3)})\right)\\ - \tfrac{5}{4}f'(a_1a_2a_3,b_1b_2b_3)\id_W\text{.} 
\end{multline}
Then $\odot_2$ is a well-defined commutative $G$-equivariant product when restricted to $\sigma(V)$.
\end{proposition}
\begin{proof}
	Note that $\cdot \odot_2 \cdot = \psi_{\id_W}\circ(\cdot \boxdot \cdot ) - \tfrac{5}{4}f'(\cdot,\cdot)\id_W$. The product is thus well-defined on $\sigma(A(G))$ and $G$-equivariant. To ensure that the product is well-defined when restricting to $\sigma(V)$, it suffices to check that the right-hand side of \cref{eq:def:odot2} is contained in the kernel of the trace map. But note that $\Tr(\cdot \odot_2|_{\sigma(V)} \cdot)$ is a $G$-invariant form, and thus equal to $f'$ up to a scalar. We prove this scalar has to be zero by the following computation, where $a,b$ are octonions with $a\cdot a=b\cdot b =0$ and $\langle a,b \rangle \neq 0$:
	\begin{multline}\label{eq:exodot2}
			\upi{a}{i}\upi{a}{i}\mathbf{1}_i\odot_2\upi{b}{i}\upi{b}{i}\mathbf{1}_i = \sum_{n,\sigma\in \mathcal{S}_3} 16 \left(((\upi{a}{i}\times \upi{b}{i})\times x_{n,\sigma(1)})((\upi{a}{i}\times e)\times(\upi{b}{i}\times e))(x_{n,\sigma(2)}\times x_{n,\sigma(3)})\right) \\-\tfrac{5}{4}f'(\upi{a}{i}\upi{a}{i}\mathbf{1}_i,\upi{b}{i}\upi{b}{i}\mathbf{1}_i)\id_W   \\
			= \langle a,b\rangle ^2 \sum_{n,\sigma\in \mathcal{S}_3} \left((\mathbf{1}_i\times x_{n,\sigma(1)})(\mathbf{1}_i)(x_{n,\sigma(2)}\times x_{n,\sigma(3)})\right)- \tfrac{5}{9}\langle a,b\rangle^2\id_W    \\
			=\langle a,b\rangle ^2 6\orth{\mathbf{1_i}} - \tfrac{5}{9}\langle a,b\rangle^2\id_W  \\
			=\langle a,b\rangle ^2\left( 6\mathbf{1}_i\mathbf{1}_j\mathbf{1}_k -\tfrac{3}{2} \sum_{r=0}^7\mathbf{1}_i{(e_r)}_i{(e_r)}_i \right) - \tfrac{5}{9}\langle a,b\rangle^2\id_W  \\
			= \langle a,b\rangle ^2(5 \mathbf{1}_i\mathbf{1}_i +\mathbf{1}_j\mathbf{1}_j+\mathbf{1}_k\mathbf{1}_k + I_i) - \tfrac{5}{9}\langle a,b\rangle^2\id_W \\
			=\langle a,b\rangle ^2(4 \mathbf{1}_i\mathbf{1}_i - I_j - I_k) + \tfrac{4}{9}\langle a,b\rangle^2\id_W\text{.}
	\end{multline}
	Applying the trace form to the right-hand side of this equation, we obtain zero. As $f'(\upi{a}{i}\upi{a}{i}\mathbf{1}_i,\upi{b}{i}\upi{b}{i}\mathbf{1}_i)$ is non-zero (see \cref{eq:bilform}), this proves $\Tr(\cdot \odot_2|_{\sigma(V)} \cdot)$ is zero, and thus $\odot_2|_{\sigma(V)}$ has image inside $\sigma(V)$.
\end{proof}

\begin{proposition}\label{prop:multbasis}
	The space of commutative $G$-equivariant products on the representation $V$ of highest weight $\omega_1+\omega_6$ is spanned by $\odot_1$ and $\odot_2$, restricted to $\sigma(V)$, defined in \cref{prop:mult1,prop:mult2}.
\end{proposition}
\begin{proof}
	The fact that these are both commutative $E_6$-equivariant products is proven in \cref{prop:mult1,prop:mult2}. The fact that they are linearly independent follows from \cref{eq:exodot1,eq:exodot2}.	
%			Where we used $\id_W = 3\big(2\mathbf{1}_1\mathbf{1}_2\mathbf{1}_3 + {(e_0)}_1{(e_0)}_2{(e_0)}_3+{(e_1)}_1{(e_3)}_2{(e_2)}_3+{(e_2)}_1{(e_6)}_2{(e_4)}_3+{(e_3)}_1{(e_5)}_2{(e_6)}_3+{(e_4)}_1{(e_1)}_2{(e_5)}_3+{(e_5)}_1{(e_2)}_2{(e_7)}_3+{(e_6)}_1{(e_7)}_2{(e_1)}_3+{(e_7)}_1{(e_4)}_2{(e_3)}_3\big)$, and $e_0$ denotes the unit of the octonion algebra. Note that the last $7$ terms enumerate all lines in the Fano plane encoding the multiplication, with a chosen start and end node.
%		These products are thus clearly linearly independent.
\end{proof}

\begin{remark}\label{noncommutative}
	Note that the space of not necessarily commutative products on the representation of highest weight $\omega_1+\omega_6$ is $3$-dimensional instead of $2$-dimensional. One basis of this space is given by $\Big\{\odot_2 = \Pi \circ \psi_{\id_W}\circ(\cdot \boxdot\cdot), \Pi \circ \psi_{\id_W}\circ(1\, 2)\circ(\cdot \boxdot\cdot),  \Pi \circ \psi_{\id_W}\circ(1\, 3)\circ(\cdot \boxdot\cdot)\Big\}$, where $(i\, j)\colon W^{\otimes3}\to W^{\otimes 3}$ is the isomorphism switching the $i$'th and $j$'th component and $\Pi\colon A(G) \to V(\omega_1+\omega_6)$ denotes the projection. In fact, one can show that $\psi_{\id_W}\circ(1\, 2)\circ(\cdot \boxdot\cdot)+  \psi_{\id_W}\circ(1\, 3)\circ(\cdot \boxdot\cdot) = 18 (\cdot \odot_1 \cdot ) - (\cdot \odot_2 \cdot)$.
\end{remark}	

\subsection{Calculating parameters}

We have finally gathered enough information to prove the main theorem of this section.
\begin{theorem}\label{identificationag2}
	Suppose $G$ is a group of type $^1E_6$ with trivial Tits algebras, and $W$ its corresponding Albert algebra. The algebra $A(G)$ is isomorphic to the subspace $\spank_k\{a_1a_2a_3\mid a_1,a_2,a_3 \in W\}\subset \End(W)$ endowed with the multiplication $\star$ given by
	\begin{multline}\label{eq:multmaintheorem}
		a_1a_2a_3 \star b_1b_2b_3 = \left(\tfrac{1}{24}\langle a_1,a_2,a_3\rangle \langle b_1,b_2,b_3 \rangle -\tfrac{1}{432}f'(a_1a_2a_3,b_1b_2b_3 )\right)\id_W
		\\ -\tfrac{1}{12}\langle a_1,a_2,a_3\rangle b_1b_2b_3-\tfrac{1}{12}\langle b_1,b_2,b_3 \rangle a_1a_2a_3\		
		\\+ \tfrac{1}{18} \left(a_1a_2a_3\right)\odot_1 \left(b_1b_2b_3\right) - \tfrac{1}{216}\left(a_1a_2a_3\right)\odot_2 \left(b_1b_2b_3\right)\text{,}
	\end{multline}
	where $\odot_1,\odot_2$ and $f'$ are defined in \cref{eq:def:odot1,eq:def:odot2,def:fprime}.	
	
\end{theorem}
\begin{proof}
	We may assume $k$ is algebraically closed.
	Multiplication is of the form \begin{multline}\label{eq:multmaintheoreme6}
		a_1a_2a_3 \star b_1b_2b_3 = \left(\tfrac{1}{9^2}\langle a_1,a_2,a_3\rangle \langle b_1,b_2,b_3 \rangle +f\left(a_1a_2a_3-\tfrac{1}{9}\langle a_1,a_2,a_3\rangle\id_W,b_1b_2b_3 -\tfrac{1}{9}\langle b_1,b_2,b_3 \rangle\id_W\right)\right)\id_W\\ +\tfrac{1}{9}\langle a_1,a_2,a_3\rangle \left(b_1b_2b_3-\tfrac{1}{9}\langle b_1,b_2,b_3\rangle\id_W\right)+\tfrac{1}{9}\langle b_1,b_2,b_3 \rangle \left(a_1a_2a_3-\tfrac{1}{9}\langle a_1,a_2,a_3\rangle\id_W\right) \\+ \left(a_1a_2a_3-\tfrac{1}{9}\langle a_1,a_2,a_3\rangle\id_W\right)\odot \left(b_1b_2b_3-\tfrac{1}{9}\langle b_1,b_2,b_3 \rangle\id_W\right)
	\end{multline}
	by \cite[Example A.6]{chayet2020class}.
	Note that $f(a_1a_2a_3,b_1b_2b_3) = \lambda f'$ for some scalar $\lambda\in k$.
	By the computations in \cref{eq:exodot1,eq:exodot2,prop:examplecomp}, we have $\odot = \tfrac{1}{18}\odot_1-\tfrac{1}{216}\odot_2$, and $\lambda = \tfrac{1}{3^3\cdot 2^4} = \tfrac{1}{432}$.

	It only remains to show that
	 $\star$ satisfies the formula above, i.e.\@ to plug in the formulas and simplify. First we will deal with the bilinear form: as $f(\cdot,\id_W)$ is a $G$-equivariant map, it should be equal to the trace form up to a scalar. The following computation determines the scalar:
	 \begin{align*}
	 	f'(a_1a_2a_3,\id_W) &= \sum_n \sum_{\tau_a,\sigma\in \mathcal{S}_3} \langle a_{\tau_a(1)} , x_{n,\sigma(2)} ,x_{n,\sigma(3)}\rangle \langle x_{n,\sigma(1)},a_{\tau_a(2)} ,a_{\tau_a(3)} \rangle \\
	 	&=\tfrac{1}{3}\sum_{n,\sigma}\sum_{\tau_a\in \mathcal{S}_3}\langle a_{\tau_a(1)} , x_{n,\sigma(2)} \times x_{n,\sigma(3)}\rangle \langle x_{n,\sigma(1)},a_{\tau_a(2)} ,a_{\tau_a(3)} \rangle \\
	 	&= \tfrac{1}{3}\sum_{b\in B}\sum_{\tau_a\in \mathcal{S}_3}6 \langle a_{\tau_a(1)} ,b \rangle \langle b,a_{\tau_a(2)} ,a_{\tau_a(3)} \rangle\\&= \tfrac{1}{3}\sum_{\tau_a\in \mathcal{S}_3}6 \langle a_{\tau_a(1)},a_{\tau_a(2)} ,a_{\tau_a(3)} \rangle \\
	 	&= 12\langle a_1,a_2,a_3\rangle.
	 \end{align*}
	 Next, we want to find simplified formulas for the maps $\cdot\odot_1 \id_W$ and $\cdot\odot_2\id_W$. As they are $G$-equivariant, we should have $\cdot\odot_i \id_W = \lambda_i \id_{\sigma(A(G))} + \mu_i \Tr(\cdot)$ for some $\lambda_i,\mu_i\in k$, where $i=1,2$. For $\odot_1$ we compute the following: 
	 \begin{align*}
	\mathbf{1}_1\mathbf{1}_2\mathbf{1}_3 \odot_1 \id_W &= \sum_{i} (\mathbf{1}_i\times (x_{\sigma(2)}\times x_{\sigma(3)}))(\mathbf{1}_i ) (x_{\sigma(1)}) -\tfrac{12}{12}\langle\mathbf{1}_1,\mathbf{1}_2,\mathbf{1}_3\rangle\id_W \\
	&= \sum_i 6\Big((\mathbf{1}_{i}\times \mathbf{1}_j)(\mathbf{1}_{i} ) (\mathbf{1}_j) +(\mathbf{1}_{i}\times \mathbf{1}_k)(\mathbf{1}_{i} ) (\mathbf{1}_k) - \tfrac{1}{4} \sum_{1\leq j\leq 8 } ( \upi{(e_j)}{i})(\mathbf{1}_{i}) (\upi{(e_j)}{i})\Big)\\
	&\hspace{5ex}-\langle\mathbf{1}_1,\mathbf{1}_2,\mathbf{1}_3\rangle\id_W \\
	&=\sum_i 6\Big(\mathbf{1}_1\mathbf{1}_2\mathbf{1}_3 + \tfrac{2}{3}\mathbf{1}_i\mathbf{1}_i+\tfrac{1}{6}I_i\Big)-\langle\mathbf{1}_1,\mathbf{1}_2,\mathbf{1}_3\rangle\id_W \\
	&= 18\mathbf{1}_1\mathbf{1}_2\mathbf{1}_3+18\mathbf{1}_1\mathbf{1}_2\mathbf{1}_3 + \id_W-\langle\mathbf{1}_1,\mathbf{1}_2,\mathbf{1}_3\rangle\id_W \\
	&=36\mathbf{1}_1\mathbf{1}_2\mathbf{1}_3+5\langle\mathbf{1}_1,\mathbf{1}_2,\mathbf{1}_3\rangle\id_W\text{.}
\end{align*}
This implies by the previous paragraph that
\begin{align*}
	 	a_1a_2a_3 \odot_1 \id_W = 36 a_1a_2a_3 +5 \langle a_1a_2a_3\rangle \id_W. 
\end{align*}
We will do a similar computation for $\cdot\odot_2\id_W$, but to do so we will use the notation introduced in \cref{not:orth}. We get	 
\begin{align*}
	\mathbf{1}_1\mathbf{1}_2\mathbf{1}_3 \odot_2 \id_W &= \sum_{i} ((\mathbf{1}_i\times y_{\sigma(1)})\times x_{\tau(1)})(\mathbf{1}_i\times( y_{\sigma(2)}\times y_{\sigma(3)} ) ) (x_{\tau(2)}\times x_{\tau(3)}) \\
	&\hspace{5ex}-15 \langle \mathbf{1}_1\mathbf{1}_2\mathbf{1}_3 \rangle \id_W \\
	&= 36\sum_i \Big(\orth(\mathbf{1}_i\times \mathbf{1}_j)+\orth(\mathbf{1}_{i}\times \mathbf{1}_k) +\tfrac{1}{2} \sum_{1\leq j\leq 8 } \orth(\mathbf{1}_i\times\upi{(e_j)}{i})\Big)\\
	&\hspace{5ex}-15\langle\mathbf{1}_1,\mathbf{1}_2,\mathbf{1}_3\rangle\id_W \\
	&=9\sum_i \Big(\orth(\mathbf{1}_k)+\orth( \mathbf{1}_j) +\tfrac{1}{2} \sum_{1\leq j\leq 8 } \orth(\upi{(e_j)}{i})\Big)\\
	&\hspace{5ex}-15\langle\mathbf{1}_1,\mathbf{1}_2,\mathbf{1}_3\rangle\id_W. 
\end{align*}
Now note that $	36 \upi{(e_j)}{i}\upi{(e_j)}{i}\mathbf{1}_i +5 \langle \upi{(e_j)}{i},\upi{(e_j)}{i},\mathbf{1}_i\rangle \id_W = \upi{(e_j)}{i}\upi{(e_j)}{i}\mathbf{1}_i \odot_1 \id_W = -12\orth(\upi{(e_j)}{i})-12\orth(\mathbf{1}_i)- \langle \upi{(e_j)}{i},\upi{(e_j)}{i},\mathbf{1}_i\rangle \id_W$ by the previous computation, so we can substitute $\orth(\upi{(e_j)}{i})$ by the equality
	\[\orth(\upi{(e_j)}{i}) = -3 \upi{(e_j)}{i}\upi{(e_j)}{i}\mathbf{1}_i -\orth(\mathbf{1}_i) +\tfrac{1}{6}\id_W.\]
	We get
	\begin{align*}
		\mathbf{1}_1\mathbf{1}_2\mathbf{1}_3 \odot_2 \id_W 
	&=9\sum_i \Big(\orth(\mathbf{1}_k)+\orth( \mathbf{1}_j) -4\orth(\mathbf{1}_i) + \tfrac{2}{3}\id_W -\tfrac{3}{2} \sum_{1\leq j\leq 8 } \upi{(e_j)}{i}\upi{(e_j)}{i}\mathbf{1}_i \Big)\\
	&\hspace{5ex}-15\langle\mathbf{1}_1,\mathbf{1}_2,\mathbf{1}_3\rangle\id_W \\
	&=9\sum_i \Big(\orth(\mathbf{1}_k)+\orth( \mathbf{1}_j) -4\orth(\mathbf{1}_i) + \tfrac{2}{3}\id_W +I_i + 4 \mathbf{1}_i\mathbf{1}_i\Big)\\
	&\hspace{5ex}-\tfrac{15}{6}\id_W \\
	&=9\Big(-2\orth(\mathbf{1}_1)-2\orth( \mathbf{1}_2) -2\orth(\mathbf{1}_3) + 3\id_W +3 \mathbf{1}_1\mathbf{1}_1 +3 \mathbf{1}_2\mathbf{1}_2 +3 \mathbf{1}_3\mathbf{1}_3\Big) -\tfrac{15}{6}\id_W\\
	&=9\Big(-12\mathbf{1}_1\mathbf{1}_2\mathbf{1}_3 -2\langle\mathbf{1}_1\mathbf{1}_2\mathbf{1}_3\rangle\id_W +3\id_W +18\mathbf{1}_1\mathbf{1}_2\mathbf{1}_3\Big) -\tfrac{15}{6}\id_W\\
	&=9(6\mathbf{1}_1\mathbf{1}_2\mathbf{1}_3  + \tfrac{8}{3}\id_W) - \tfrac{15}{6}\id_W\\
	&=54\mathbf{1}_1\mathbf{1}_2\mathbf{1}_3  +129\langle \mathbf{1}_1,\mathbf{1}_2,\mathbf{1}_3 \rangle\id_W.
	\end{align*}
	 This now implies in the same way as before that
	 \begin{align*}
	 	a_1a_2a_3 \odot_2 \id_W = 54 a_1a_2a_3 +129 \langle a_1a_2a_3\rangle \id_W .
	 \end{align*}
	 Combining the expression for $\cdot\odot_1 \id_W$ and $\cdot\odot_1 \id_W$ we get
	 \begin{align*}
	 	a_1a_2a_3 \odot \id_W = \tfrac{7}{4} a_1a_2a_3 -\tfrac{23}{72} \langle a_1a_2a_3\rangle \id_W.
	 \end{align*}
	 Now all that remains is to plug in all the bits of information we have gathered. This results in				 
	 %\todoi{CHECK FROM HERE}		 
%	 \begin{multline}\label{eq:multmaintheoreme6}
%		a_1a_2a_3 \star b_1b_2b_3 = \left(-\tfrac{1}{9^2}\langle a_1,a_2,a_3\rangle \langle b_1,b_2,b_3 \rangle +f(a_1a_2a_3,b_1b_2b_3 )\id_W\right.
%		\\+f\left(a_1a_2a_3, -\tfrac{1}{9}\langle b_1,b_2,b_3 \rangle\id_W\right)
%		\\+f\left(-\tfrac{1}{9}\langle a_1,a_2,a_3\rangle\id_W,b_1b_2b_3 \right)
%		\\\left.+f\left(-\tfrac{1}{9}\langle a_1,a_2,a_3\rangle\id_W, -\tfrac{1}{9}\langle b_1,b_2,b_3 \rangle\id_W\right)\right)\id_W
%		\\ +\tfrac{1}{9}\langle a_1,a_2,a_3\rangle \left(b_1b_2b_3\right)+\tfrac{1}{9}\langle b_1,b_2,b_3 \rangle \left(a_1a_2a_3\right) 		
%		\\+ \left(a_1a_2a_3\right)\odot \left(b_1b_2b_3\right)
%		\\+\left(a_1a_2a_3\right)\odot \left(-\tfrac{1}{9}\langle b_1,b_2,b_3 \rangle\id_W\right)
%		\\+\left(-\tfrac{1}{9}\langle a_1,a_2,a_3\rangle\id_W\right)\odot \left(b_1b_2b_3\right)
%		\\+\left(-\tfrac{1}{9}\langle a_1,a_2,a_3\rangle\id_W\right)\odot \left(-\tfrac{1}{9}\langle b_1,b_2,b_3 \rangle\id_W\right)\text{,}
%	\end{multline}
	\begin{multline}\label{eq:multmaintheoreme7}
		a_1a_2a_3 \star b_1b_2b_3 = \left(-\tfrac{1}{9^2}\langle a_1,a_2,a_3\rangle \langle b_1,b_2,b_3 \rangle +f(a_1a_2a_3,b_1b_2b_3 )\id_W\right.
		\\\left.-\tfrac{12}{9\cdot 432}\langle b_1,b_2,b_3 \rangle \langle a_1,a_2,a_3\rangle
		\right)\id_W
		\\ +\tfrac{1}{9}\langle a_1,a_2,a_3\rangle \left(b_1b_2b_3\right)+\tfrac{1}{9}\langle b_1,b_2,b_3 \rangle \left(a_1a_2a_3\right) 		
		\\+ \left(a_1a_2a_3\right)\odot \left(b_1b_2b_3\right)
		\\-\tfrac{1}{9}\langle b_1,b_2,b_3 \rangle\left(\tfrac{7}{4} a_1a_2a_3 -\tfrac{23}{72} \langle a_1a_2a_3\rangle \id_W \right)
		\\-\tfrac{1}{9}\langle a_1,a_2,a_3\rangle \left(\tfrac{7}{4} b_1b_2b_3 -\tfrac{23}{72} \langle b_1b_2b_3\rangle \id_W \right)		
		\\-\tfrac{1}{72}\langle a_1,a_2,a_3\rangle\langle b_1,b_2,b_3 \rangle\id_W\text{.}
	\end{multline}
	Simplifying this gives the formula in the statement of this theorem. \qedhere
%	 \begin{multline}\label{eq:multmaintheoreme7}
%		a_1a_2a_3 \star b_1b_2b_3 = \left(\tfrac{1}{24}\langle a_1,a_2,a_3\rangle \langle b_1,b_2,b_3 \rangle +f(a_1a_2a_3,b_1b_2b_3 )\right)\id_W
%		\\ -\tfrac{1}{12}\langle a_1,a_2,a_3\rangle \left(b_1b_2b_3\right)-\tfrac{1}{12}\langle b_1,b_2,b_3 \rangle \left(a_1a_2a_3\right) 		
%		\\+ \left(a_1a_2a_3\right)\odot \left(b_1b_2b_3\right)\text{,}\qedhere
%	\end{multline}		 
\end{proof} 
\section{The automorphism group}\label{sec:autgroup}
In \cite{chayet2020class}, it was noted that the automorphism group of $A(G)$ with $G$ of type $E_6$ is smooth with identity component $G/Z(G)$. By combining this with \cref{thm:extension,prop:transposeisouterauto}, we can determine what the full automorphism group is. Indeed, to find the stabiliser of a $G$-equivariant multiplication, we only need to check that the multiplication is equivariant under transposition (regarded as an operator on $A(G)\subseteq \End(V(\omega_1))$. This will automatically be the case.
	\begin{proposition}
		The algebra $A(G)$ has automorphism group $G/Z(G)\rtimes C_2$.
	\end{proposition}
	\begin{proof}
		The $C_2$ symmetry of the Dynkin diagram of $E_6$ induces an automorphism of the Lie algebra $\mathfrak{e}_6$. This automorphism acts non-trivially on $A(E_6)$, as it acts non-trivially on the set of weights. This means the outer automorphism on the Lie algebra induces a non-trivial automorphism of $A(E_6)$, so by \cref{thm:extension}, this is also the full automorphism group.
	\end{proof}
	
	In fact, we can say more than this.
	\begin{theorem}\label{thm:allmultc2}
	Let $G$ be an adjoint group of type $E_6^{\mathrm{ad}}$.
		Any $G$-equivariant commutative algebra structure on $V(\omega_1+\omega_6)$ has automorphism group equal to $G\rtimes C_2$. 
	\end{theorem}
	\begin{proof}
		A commutative $G$-equivariant algebra structure on $V(\omega_1+\omega_6)$ has multiplication of the form $\lambda \odot_1+\mu \odot_2$, by \cref{prop:multbasis}. The automorphism group of such an algebra is smooth with identity component $G/Z(G)$ by \cite[Lemma~5.1]{GG15}. By \cref{thm:extension}, the full automorphism group can then only be extended by the transposition operator $ ^\top$. Since the commutative multiplication space is $2$-dimensional, the map $\Hom_{E_6}(\Sq^2V,V) \to \Hom_{E_6}(\Sq^2 V,V) \colon f\mapsto f^\top$, where $f^\top(v\otimes w) \coloneqq f(v^\top\otimes w^\top)^\top$, is (with respect to a proper basis) one of the following:
		\[ \begin{pmatrix}
			1 & 0\\
			0 & 1
		\end{pmatrix}, \begin{pmatrix}
			-1 & 0\\
			0 & 1
		\end{pmatrix},
		\begin{pmatrix}
			-1 & 0\\
			0 & -1
		\end{pmatrix},
		\begin{pmatrix}
			1 & 1\\
			0 & 1
		\end{pmatrix},
		\begin{pmatrix}
			-1 & 1\\
			0 & -1
		\end{pmatrix}  \] 
		Because it has at most order $2$, it cannot be one of the last two. If it is not equal to the identity, it should have a $(-1)$-eigenvector. We will show this is an impossibility. Suppose $\lambda \odot_1+\mu \odot_2$ is a $(-1)$-eigenvector.
		Note that $\upi{a}{i}\upi{a}{i}e\odot_1^\top \upi{b}{i}\upi{b}{i}e = \upi{a}{i}\upi{a}{i}e\odot_1\upi{b}{i}\upi{b}{i}e$ and $\upi{a}{i}\upi{a}{i}e\odot_2^\top \upi{b}{i}\upi{b}{i}e =\upi{a}{i}\upi{a}{i}e\odot_2\upi{b}{i}\upi{b}{i}e$, so  $\upi{a}{i}\upi{a}{i}e(\lambda \odot_1^\top+\mu \odot_2^\top ) \upi{b}{i}\upi{b}{i}e =0$. But this is impossible, since $\upi{a}{i}\upi{a}{i}e\odot_1\upi{b}{i}\upi{b}{i}e$ and $\upi{a}{i}\upi{a}{i}e\odot_2\upi{b}{i}\upi{b}{i}e$ are linearly independent in $V(\omega_1+\omega_6)$.		
	\end{proof}
	\begin{remark}\label{autnoncomm}
		Interestingly, \cref{thm:allmultc2} is no longer true when we omit the word ``commutative''. In this case, the anticommutative multiplication $\psi_{\id_W}\circ(1\, 2)\circ(\cdot \boxdot\cdot)-  \psi_{\id_W}\circ(1\, 3)\circ(\cdot \boxdot\cdot)$ as defined in \cref{noncommutative}, is a $(-1)$-eigenvector for the outer automorphism. Thus any non-commutative multiplication on $V(\omega_1+\omega_6)$ has as automorphism group precisely $G$.
	\end{remark}
	%\nocite{*}
\bibliographystyle{alpha}
\bibliography{sources.bib}

\end{document}